\documentclass{amsart}
 
\usepackage{amssymb, array, verbatim, amscd}
\usepackage{amsmath, amsfonts,amsthm}
\usepackage{color}

\theoremstyle{plain}

\theoremstyle{definition}

\newcommand{\calP}{\mathcal {P}}

\newcommand{\calC}{\mathcal {C}}

\newcommand{\End}{\operatorname{End}}
\newcommand{\Hom}{\operatorname{Hom}}

\newcommand{\rk}{\operatorname{rank}}

\numberwithin{equation}{section}
\newtheorem{thm}[equation]{Theorem}
\newtheorem{cor}[equation]{Corollary}
\newtheorem{lem}[equation]{Lemma}

\newtheorem{defn}[equation]{Definition}

\newtheorem{ex}[equation]{Example}

\newcommand{\exref}[1]{Example~\ref{#1}}
\newcommand{\lemref}[1]{Lemma~\ref{#1}}

\newcommand{\thmref}[1]{Theorem~\ref{#1}}

\newcommand{\raro}{\rightarrow}

\newcommand{\dbf}{\itshape\bfseries}   % emphasis for definitions

\newcommand{\Ker}{\operatorname{Ker}}
\newcommand{\tp}{\operatorname{type}}

\newcommand{\nriso}{\cong_{\text{\rm nr}}}  % near--isomorphism
\newcommand{\Z}{{\mathbb Z}}
\newcommand{\Q}{{\mathbb Q}}

\newcommand{\Tc}{\operatorname{T_{cr}}}
\newcommand{\R}{\operatorname{R}}

 \newcounter{noteno}

\begin{document}

\title[Main  Decomposition]{Completely decomposable direct
  summands of torsion-free abelian groups of finite rank}

\author {Adolf Mader}
\address {Department of Mathematics \\
University of Hawaii at Manoa \\
2565 McCarhty Mall, Honolulu, HI 96922, USA}
\email{adolf@math.hawaii.edu}
  
\author {Phill Schultz}
\address {School of Mathematics and Statistics\\
The University of Western Australia \\ Nedlands\\
 Australia,  6009}
\email {phill.schultz@uwa.edu.au}
 
\maketitle

 %%%%%%%%%%%%%%%%%%%%%%%%%%%%%%%%%
\begin{abstract} Let $A$ be a finite rank torsion--free abelian group.
  Then there exist direct decompositions $A=B\oplus C$ where $B$ is
  completely decomposable and $C$ has no rank 1 direct summand. In
  such a decomposition $B$ is unique up to isomorphism and $C$ unique
  up to near--isomorphism.
\end{abstract} 

 \section{Introduction} 

 Torsion-free abelian groups of finite rank (tffr groups) are best
 thought of as additive subgroups of finite dimensional $\Q$--vector
 spaces. All ``groups'' in this article are torsion-free abelian
 groups of finite rank. The rank of a group $A$ is the dimension of
 the vector space $\Q A$ that $A$ generates. By reason of rank, such
 groups always have ``indecomposable decompositions'', meaning direct
 decompositions with indecomposable summands. Although
   as shown in \cite{Lady74}, a group has only finitely many
   non--isomorphic summands, its indecomposable decompositions can be
   highly non--unique, (see for example \cite[Section~90]{F2}), and a
 group may have such decompositions in which the number of summands or
 the ranks of the summands differ. A particularly striking result in
 this direction is due to A.L.S.  Corner \cite{Cr61}, \cite{Cr69N}.

 Let $P = (r_1, \ldots, r_t)$ be a partition of $n$, i.e., $r_i \geq
 1$ and $r_1 + \cdots + r_t = n$. Then {\dbf $G$ realizes $P$} if
 there is an indecomposable decomposition $G = G_1 \oplus \cdots
 \oplus G_t$ such that for all $i,\  r_i = \rk(G_i)$.

 {\bf Corner's Theorem.} {\em Given integers $n \geq k \geq 1$, there
   exists a group $G$ of rank $n$ such that $G$ realizes every
   partition of $n$ into $k$ parts $n = r_1 + \cdots + r_k$.}
   
 Corner's Theorem is related to two problems posed by Fuchs
 \cite[Problems 67 and 68]{F2}, namely
\begin{enumerate}
\item Given an integer $m$, find all sequences $n_1<\cdots<n_s$ for
  which there is a tffr group of rank $m$ having indecomposable
  decompositions into $n_1,\cdots, n_s$ summands,
\item Given partitions $r_1+\cdots + r_k=n=r_1'+\cdots +r'_\ell$ of a
  positive integer $n$, under what conditions does there exist a tffr
  group with indecomposable decompositions with summands of ranks
  $r_1,\dots,r_k$ and $r'_1,\dots r'_\ell$ respectively?
\end{enumerate}
The second problem of Fuchs was solved by Blagoveshchenskaya,
\cite[Theorem 13.1.19]{M00} for a restricted class $\calC$ of groups:
let $P$ and $Q$ be partitions of $n$. There is a group $G\in\calC$
realising $P$ and $Q$ if and only if the sum of the largest part of
each and the number of parts of the other does not exceed $n+1$.

More generally, one can pose the

 {\bf Question:} {\em Characterize the families $\calP$ of partitions
   of $n$ that can be realized by a tffr group.}

 Corner's Theorem shows that families of partitions of $n$ of fixed
 length $k$ can be realized. On the other hand, he comments that
   \begin{quotation}\dots it can be shown quite readily that an
     equation such as $1+1+2=1+3$ \dots 
     cannot be realized.
   \end{quotation} 

A more general question was settled by Lee Lady for almost completely
decomposable groups (defined below)
\cite[Corollary~7]{La74a}, \cite[Theorem~9.2.7]{M00}. A group $G$ is
{\dbf clipped} if it has no direct summands of rank $1$. Lady's ``Main
Decomposition Theorem'' says that every almost completely decomposable
group $G$ has a decomposition $G = G_{cd} \oplus G_{cl}$ where
$G_{cd}$ is completely decomposable, $G_{cl}$ is clipped, $G_{cd}$ is
unique up to isomorphism, and $G_{cl}$ is unique up to
near--isomorphism.  Near isomorphism is a weakening of isomorphism due
to Lady \cite{Lady75}. There are several equivalent definitions, see
for example \cite[Chapter~9]{M00}, the most useful one for us being
that a group $A$ is nearly isomorphic to $B$, denoted $A \nriso B$,
if there exists a group $K$ such that $A \oplus K \cong B \oplus K$.
 
It follows from this definition that near isomorphism is an
equivalence relation on the class of groups. Moreover rank and the
property of being clipped are invariants of near isomorphism classes.

An important result due to Arnold \cite[12.9]{Arnold82},
\cite[Theorem~12.2.5]{M00}, is that if $A \nriso A'$ and $A = X \oplus
Y$, then $A' = X' \oplus Y'$ with $X \nriso X'$ and $Y \nriso Y'$.
Conversely, if $X \nriso X'$ and $Y \nriso Y'$, then $X \oplus Y
\nriso X' \oplus Y'$.

Let $A$ be a group. We say that an indecomposable decomposition $A =
\bigoplus_{i\in[n]} A_i$ of $A$ is {\dbf unique up to near
  isomorphism} if whenever $A = \bigoplus_{j\in[m]} B_j$ is an
indecomposable decomposition of $A$, then $n=m$ and there is a
permutation $\sigma$ of $[n]$ such that $A_i \nriso B_{\sigma(i)}$ for
all $i\in[n]$.

By Arnold's Theorem, nearly isomorphic groups of rank $n$ realize the
same partitions of $n$.

Denote the partition $(m, 1,\dots, 1)$ where there are $k$ 1s, by
$(m,1^k)$. Since indecomposable groups are certainly clipped, if an
almost completely decomposable group of rank $n$ realizes partitions
$(m, 1^{n-m})$ and $(m', 1^{n-m' })$, then $m=m'$.
 
Our main result is the generalization of the Main Decomposition
Theorem to arbitrary torsion-free groups of finite rank (\thmref{main
  decomposition}) which then settles Corner's remark.
  
 It may be asked to describe the isomorphism classes of indecomposable
 groups of a given rank. Rank--$1$ groups are indecomposable and have
 been classified by means of types (\cite{Lev19}, \cite{F2}) and there
 are $2^{\aleph_0}$ isomorphism classes. It is also possible to
 describe the indecomposable almost completely decomposable groups of
 rank $2$ (see \cite[Section 12.3]{M00}) but in general this task must
 be accepted as being hopeless.

 A {\dbf completely decomposable} group is a direct sum of rank--$1$
 groups, and completely decomposable groups were classified in terms
 of cardinal invariants by Baer \cite[Section 86, page 113]{F2}.  In
 particular, their decompositions into rank--{1} summands are unique
 up to isomorphism.

{\dbf Almost completely decomposable groups} are finite
extensions of completely decomposable groups of finite rank. This
class of groups was introduced and first studied by Lee Lady
\cite{La74a}, see \cite{M00} for a comprehensive exposition. An almost
completely decomposable group $X$ contains special completely
decomposable subgroups, namely those of minimal index in $X$, the
{\dbf regulating subgroups of $X$}. Rolf Burkhardt \cite{Bt84} showed
that the intersection of all regulating subgroups is again a
completely decomposable subgroup of finite index in $X$. This group,
that is fully invariant in $X$, is the {\dbf regulator} $\R(X)$ of
$X$. 

Most published examples of groups with non--unique decompositions are
almost completely decomposable groups. It is also noteworthy that for
an almost completely decomposable group $X$ with non--unique
indecomposable decompositions the index $[X : \R(X)]$ is a composite
number. On the other hand if $[X : \R(X)]$ is the power of a prime
$p$, then Faticoni and Schultz proved that the indecomposable
decompositions of $X$ are unique up to near--isomorphism,
\cite{FaSchu96}, \cite[Corollary~10.4.6]{M00}.  The problem then
remains to determine the near--isomorphism classes of indecomposables.
For an almost completely decomposable group $X$, write $\R(X) =
\bigoplus_{\rho \in \Tc(X)} R_\rho$ with $\rho$--homogeneous
components $R_\rho$ and $R_\rho \neq 0$. Then $\Tc(X)$ is called the
{\dbf critical typeset of~$X$.}  The problem has been largely solved
when the critical typeset is an inverted forest in a number of papers
by Arnold--Mader--Mutzbauer--Solak (\cite{AMMS12}, \cite{AMMS13a},
\cite{AMMS13b}, \cite{AMMS14a}, \cite{AMMS14b}, \cite{AMMS15a},
\cite{AMMS15b}, \cite{AMMS16a}, \cite{AMMS16b}) using representations
of posets as a tool.

\section{Main Decomposition}

A {\dbf rank--$1$ group} is a group isomorphic with an additive
subgroup of $\Q$. A {\dbf type} is the isomorphism class of a
$\rk$--$1$ group. It is easy to see that every $\rk$--$1$ group is
isomorphic to a {\dbf rational group} by which we designate an
additive subgroup of $\Q$ that contains $1$. If $A$ is a $\rk$--$1$
group, then $\tp(A)$ denotes the type of $A$, i.e., the isomorphism
class containing $A$. Types are commonly denoted by $\sigma, \tau,
\ldots$. We will also use $\sigma, \tau, \ldots$ to mean a rational
group of type $\sigma, \tau, \ldots$. It will always be clear from the
context whether $\tau$ is a rational group or a type. The advantage is
that any completely decomposable group $A$ of finite rank $r$ can be
written as $A = \sigma v_1 \oplus \cdots \oplus \tau_r v_r$ with $v_i
\in A$ because $1 \in \tau_i$, and $\tp(\tau_i) = \tau_i$. In this
case $\{v_1, \ldots, v_r\}$ is called a {\dbf decomposition basis} of
$A$. 
 
A completely decomposable group is called {\dbf $\tau$--homogeneous}
if it is the direct sum of rank--$1$ groups of type $\tau$, and {\dbf
  homogeneous} if it is $\tau$--homogeneous for some type $\tau$. It
is known \cite[86.6]{F2} that pure subgroups of homogeneous completely
decomposable groups are direct summands.

\begin{defn} A  group $G$ is {\dbf $\tau$--clipped} if $G$ does
  not possess a rank--$1$ summand of type $\tau$.
\end{defn}

\begin{lem}\label{tau homogeneous case} Suppose that $G = D \oplus B = A
  \oplus C$ where $B$ and $C$ are $\tau$--clipped and $D, A$ are
  completely decomposable and $\tau$--homogeneous. Then $D \cong A$.
\end{lem} 

\begin{proof} Let $\delta, \beta, \alpha, \gamma \in \End(G)$ be the 
  projections belonging to the given decompositions. Let $0 \neq x \in
  D$. Then $x = x\alpha + x\gamma$. Assume that $x \alpha = 0$. Then
  $\langle x \rangle_*$ is a pure rank--$1$ subgroup of $D$ and hence
  a summand of $D$ and of $G$. Also $\langle x\rangle_* \alpha = 0$
  which says the $\langle x\rangle_* \subseteq \Ker \alpha = C$. It
  follows that $\langle x \rangle_*$ is a rank--$1$ summand of $C$ of
  type $\tau$, contradicting the fact that $C$ is $\tau$--clipped.
  Hence $\alpha : D \rightarrow A$ is a monomorphism and therefore
  $\rk D \leq \rk A$. By symmetry $\rk A \leq \rk D$ and $D \cong A$
  as desired.
\end{proof}

The direct sum of $\tau$--clipped groups need not be $\tau$--clipped
as \exref{fail sum of tau clipped} shows.

\begin{ex}\label{fail sum of tau clipped} Let $p, q$ be different
  primes and let $\sigma, \tau$ be rational groups that are
  incomparable as types and such that neither $\frac{1}{p}$ nor
  $\frac{1}{q}$ is contained in either $\sigma$ or $\tau$. Let
  \[
  X_1 = (\sigma v_1 \oplus \tau v_2) + \Z \frac{1}{p} (v_1 + v_2)
  \text{ and }
  X_2 = (\sigma w_1 \oplus \tau w_2) + \Z \frac{1}{q} (w_1 + w_2).
  \]  
  It is easy to see that $\R(X_1) = \sigma v_1 \oplus \tau v_2$ and
  $\R(X_2) = \sigma w_1 \oplus \tau w_2$, and that $X_1$ and $X_2$ are
  indecomposable and, in particular, clipped. There exist integers
  $u_1, u_2$ such that $u_1 p + u_2 q = 1$. Now $\frac{1}{p} (v_1 +
  v_2) + \frac{1}{q} (w_1 + w_2) = \frac{1}{p q} \left((q v_1 + p w_1)
    + (q v_2 + p w_2)\right)$. Set $v_1' = q v_1 + p w_1$, $v_2' = q
  v_2 + p w_2$, $w_1' = -u_1 v_1 + u_2 w_1$, and $w_2' = -u_1 v_2 +
  u_2 w_2$. Then (change of decomposition basis) $\sigma v_1 \oplus
  \sigma w_1 = \sigma v_1' \oplus \sigma w_1'$ and $\tau v_2 \oplus
  \tau w_2 = \tau v_2' \oplus \tau w_2'$. Hence $X = \left(\sigma w_1'
    \oplus \tau w_2'\right) \oplus \left(\left(\sigma v_1' \oplus \tau
      v_2'\right) + \Z \frac{1}{p q} (v_1' + v_2')\right)$ so $X$ has
  rank--$1$ summands of type $\sigma$ and $\tau$.
\end{ex} 

However, \lemref{sum of tau clipped} settles positively a special case.

\begin{lem}\label{sum of tau clipped} Let $G = A \oplus B$ where $A =
  \bigoplus_{\rho \neq \tau} A_\rho$ is completely decomposable and
  $B$ is $\tau$--clipped. Then $G$ is $\tau$--clipped.
\end{lem}

\begin{proof} We may assume that $\rk A = 1$. In fact, if $A
  = A_1 \oplus \cdots \oplus A_k$ where $\rk A_i = 1$, then $A_k
  \oplus B$ is $\tau$--clipped by the rank $1$ case, $A_2 \oplus
  \cdots \oplus A_k \oplus B$ is $\tau$--clipped by induction, and $A
  \oplus B$ is $\tau$--clipped by the rank $1$ case. 

  By way of contadiction assume that $G = \tau v \oplus C = \sigma a
  \oplus B$ with $\tau \not\cong \sigma$ (as rational groups or $\tau
  \neq \sigma$ as types). Let $\alpha : G \raro \sigma a \subseteq G$,
  $\beta : G \raro B \subseteq G$, $\delta : G \raro \tau v \subseteq
  G$, and $\gamma : G \raro C \subseteq G$ be the projections
  (considered endomorphisms of $G$) that come with the stated
  decompositions. 

\begin{enumerate}
\item We have $v = v\alpha + v\beta$ uniquely. Suppose $v \alpha = 0$.
  Then $(\tau v) \alpha = 0$ and the summand $\tau v$ is contained in
  $\Ker \alpha = B$. Then $\tau v$ is a summand of $B$ contradicting
  the fact that $B$ is $\tau$--clipped. So $\alpha : \tau v \raro
  \sigma a$ is a monomorphism and $\tau \leq \sigma$.
\item We have $a = a\delta + a\gamma$. Suppose that $a \delta = 0$.
  Then $(\sigma a) \delta = 0$ and the summand $\sigma a$ is contained
  in $\Ker \delta = C$. Hence $C = \sigma a \oplus C'$ for some $C'$
  and $G = \tau v \oplus \sigma a \oplus C' = \sigma a \oplus
  B$. Hence $\frac{G}{\sigma a} \cong \tau v \oplus C' \cong B$. This
  contradicts the fact that $B$ is $\tau$--clipped. So $\delta :
  \sigma a \raro \tau v$ is a monomorphism and hence $\sigma \leq
  \tau$. 
\item By (1) and (2) we get the contradiction $\sigma = \tau$, saying
  that $G = \sigma a \oplus B$ does not have a rank--$1$ summand of
  type $\tau$, and the special case is proved.
\end{enumerate} 
\end{proof} 

\begin{thm}\label{main decomposition} {\rm\bf (Main Decomposition.)}
  Let $G$ be a torsion-free group of finite rank. Then there are
  decompositions $G = A_0 \oplus A_1$ in which $A_0$ is completely
  decomposable and $A_1$ is clipped. 

  Suppose that $G = A_0 \oplus A_1 = B_0 \oplus B_1$
  where $A_0$ and $B_0$ are completely decomposable and $A_1$ and
  $B_1$ are clipped. Then $A_0 \cong B_0$ and consequently $A_1\nriso
  B_1$
  \end{thm}

  \begin{proof} Let $A_0$ be a completely decomposable summand of $G$
    of maximal rank. Then $G = A_0 \oplus A_1$ and $A_1$ is clipped.

    Let $A_0 = \bigoplus_{\rho} A_{\rho}$ and $B_0 =
    \bigoplus_{\rho} B_{\rho}$ be the homogeneous decompositions of
    the completely decomposable groups $A_0$ and $B_0$. By allowing
    $A_{\rho}$ and $B_{\rho}$ to be the zero group, we may assume that
    the summation index ranges over all types $\rho$.
  
  We consider $G = A_\tau \oplus \left(\bigoplus_{\rho \neq
      \tau} A_\rho \oplus A_1\right) = B_\tau \oplus
  \left(\bigoplus_{\rho \neq \tau} B_\rho + B_1\right)$. By \lemref{sum
    of tau clipped} $\bigoplus_{\rho \neq \tau} A_\rho \oplus A_1$ and
  $\bigoplus_{\rho \neq \tau} B_\rho + B_1$ are both $\tau$--clipped.
  Hence by \lemref{tau homogeneous case} we conclude that $A_\tau \cong
  B_\tau$. Here $\tau$ was an arbitrary type and the claim is
  clear. The fact that $A_1\nriso B_1$ follows from the
  isomorphism $A_0\oplus A_1\cong A_0\oplus B_1$. 
\end{proof} 

\begin{cor}\label{realizing hooks} Suppose that $G$ has rank
  $n$ and $G$ realizes the partitions  
 $(m, 1^{n-m})$ and $ (m',1^{n-m'})$ 
  Then $m=m'$.
\end{cor}

\begin{proof} The indecomposable summands of ranks $m$ and $m'$ are
  necessarily clipped. so by Theorem \ref{main decomposition}, the
  completely decomposable parts of the decompositions are isomorphic.
\end{proof}

In particular there is no group that realizes both $(1,1,2)$ and
$(1,3)$. 

We call a decomposition $G = G_{cd} \oplus G_{cl}$ with $G_{cd}$
completely decomposable and $G_{cl}$ clipped a {\dbf Main
  Decomposition of $G$}. 

\begin{cor}\label{cdsummands} Let $C$ be a completely decomposable
  direct summand of a group $G$.  Then $G$ has a Main Decomposition
  $G_{cd}\oplus G_{cl}$ in which $C$ is a direct summand of $G_{cd}$.
\end{cor}

\begin{proof} Let $G = C \oplus B$ and let $B$ have Main Decomposition
  $B=B_{cd}\oplus B_{cl}$. Then $G=(C\oplus B_{cd})\oplus B_{cl}$ is a
  Main Decomposition of $G$.
\end{proof}

Main Decompositions are unique only up to near isomorphism. For
example, let $X = \tau v \oplus \left((\tau v_1 \oplus \sigma v_2) +
  \Z\frac{1}{5} (v_1 \oplus v_2)\right)$. The group $(\tau v_1 \oplus
\sigma v_2) + \Z\frac{1}{5} (v_1 \oplus v_2)$ is indecomposable, hence
clipped. We also have $X = \tau(v+v_1) \oplus \left((\tau v_1 \oplus
  \sigma v_2) + \Z\frac{1}{5} (v_1 \oplus v_2)\right)$ and $\tau v
\neq \tau (v+v_1)$. On the other hand if $G = G_{cd} \oplus G_{cl}$
and $\Hom(G_{cd}, G_{cl}) = 0$, then $G_{cd}$ is unique and direct
complements of $G_{cd}$ are isomorphic (\cite[Lemma~1.1.3]{M00}).

%%%%%%%%%%%%%%%%%%%%%%%%%%%%

\end{document}